\documentclass[11pt]{article}
\usepackage[T1]{fontenc}
\usepackage[utf8]{inputenc}
\usepackage{geometry}
\usepackage{amsmath}
\usepackage{amsthm}
\usepackage{enumitem}
\usepackage{setspace}
\usepackage[hidelinks]{hyperref}
\usepackage[dvipsnames]{xcolor} 
\usepackage[color=Purple!50!white, textwidth=22mm]{todonotes}
\usepackage{comment}
\allowdisplaybreaks

\makeatletter
\newtheorem{proposition}{Proposition}
\newtheorem{theorem}{Theorem}

\newtheorem{innercustomgeneric}{\customgenericname}
\providecommand{\customgenericname}{}
\newcommand{\newcustomtheorem}[2]{%
  \newenvironment{#1}[1]
  {%
   \renewcommand\customgenericname{#2}%
   \renewcommand\theinnercustomgeneric{##1}%
   \innercustomgeneric
  }
  {\endinnercustomgeneric}
}

\newcustomtheorem{customthm}{Theorem}
\newcustomtheorem{customlemma}{Lemma}

\usepackage{amssymb}

\makeatother

\setenumerate[0]{label=(\roman*), ref=\roman*}

\begin{document}
\title{A note on the Alon-Saks-Seymour problem}
\author{Jacob Fox\thanks{Department of Mathematics, Stanford University, Stanford, CA 94305. Email: {\tt jacobfox@stanford.edu}. Research supported by NSF awards DMS-2452737 and DMS-2154129.}}
\date{}
\maketitle

\begin{abstract}
Let $f(k)$ be the maximum possible chromatic number of a graph whose edge set can be partitioned into at most $k$ complete bipartite graphs. Alon, Saks, and Seymour conjectured that $f(k)=k+1$ for all $k$. While the conjecture was verified for $k \leq 9$ by Gao et al., it was disproved by Huang and Sudakov, and further Balodis et al. proved that $f(k) \geq 2^{\widetilde{\Omega}((\log k)^2)}$. 

In this note, we give a simple proof of the recursive upper bound $f(k+1) \leq f(k)+f(\lfloor k/4 \rfloor)$. Consequently, $f(k) \leq 2^{(\log_2 (4k))^2/4}$ for $k \geq 1$. This improves the previous best known upper bound of Mubayi and Vishwanathan in the exponent by a factor which is asymptotically two. Note that these bounds are sharp up to a lower order factor in the exponent by the result of Balodis et al. 

\end{abstract}

\section{Introduction}

For a positive integer $k$, let $f(k)$ be the maximum possible chromatic number of a graph whose edge set can be partitioned into at most $k$ complete bipartite graphs. It is also convenient to set $f(0)=1$. As a generalization of the Graham-Pollak theorem, Alon, Saks, and Seymour conjectured that $f(k)=k+1$ for all $k$. Gao, McKay, Naserasr, and Stevens \cite{GMNS} verified the Alon-Saks-Seymour conjecture for $k \leq 9$. 

If graphs $G,G_1,G_2$ satisfy $E(G) = E(G_1) \cup E(G_2)$, then $\chi(G) \leq \chi(G_1)\chi(G_2)$. It easily follows that $f(k+1) \leq 2f(k)$ and inductively $f(k) \leq 2^k$. This was improved by Mubayi and Vishwanathan \cite{MV}, who proved that $f(k)$ is  asymptotically at most $ 2^{((\log k)^2+\log k)/2}$ (all logarithms in this paper are base $2$).

The Alon-Saks-Seymour conjecture was disproved by Huang and Sudakov \cite{HS}. They showed that $f(k) \geq ck^{6/5}$ for an appropriate constant $c$. The first super-polynomial lower bound on $f(k)$ was proved by G\"o\"os \cite{Goos}. This was subsequently improved by Balodis, Ben-David, G\"o\"os, Jain, and Kothari \cite{BBGJK} to $f(k) \geq 2^{\widetilde{\Omega}((\log k)^2)}$, which is sharp up to the lower order factor in the exponent. The Alon-Saks-Seymour problem has found interesting connections to communication complexity and learning theory (see \cite{BBGJK,HS}). 

Here we prove the following recursive upper bound. It improves the bound of Mubayi and Vishwanathan in the exponent by a factor $2$. 

\begin{theorem}\label{main}
We have $f(k+1) \leq f(k)+f(\lfloor k/4 \rfloor)$. Consequently, $f(k) \leq 2^{(\log_2 (4k))^2/4}$ for $k \geq 1$. 
\end{theorem}

The proof of the latter bound in Theorem \ref{main} follows from the former recursive bound by strong induction on $k$, with the base cases $1 \leq k \leq 3$ being easy to check. For $k \geq 4$, by applying the recursive inequality for $k-1,\ldots,\lfloor k/4 \rfloor$ and using monotonicity of $f$, we get $f(k) \leq (\lceil 3k/4 \rceil +1)f(\lfloor k/4 \rfloor) \leq kf(\lfloor k/4 \rfloor)$. Then, for $k \geq 4$ we get by the induction hypothesis that $$f(k) \leq kf(\lfloor k/4 \rfloor) \leq k2^{(\log k)^2/4} = 2^{\log k + (\log k)^2/4} \leq 2^{(\log(4k))^2 /4}.$$
For brevity, we omit the proof that the recursive bound implies $f(k) \leq 2^{\frac{1}{4}(\log k)^2 - \frac{1}{2}(\log k)\log \log k +O(\log k)}$.

We are left to prove the first part of Theorem \ref{main}. As a warmup, we first give a very simple proof of a weaker recursive bound which can recover the Mubayi-Vishwanathan bound. 

\begin{proposition}\label{easier}
We have $f(k+1) \leq f(k)+f(\lfloor k/2 \rfloor)$. \end{proposition}
\begin{proof}
Let $G=(V,E)$ be a graph whose edge set can be partitioned into $k+1$ complete bipartite graphs. Let $A,B$ denote the parts of one of these complete bipartite graphs. Every other complete bipartite graph in the edge-partition cannot have  an edge internal to $A$ and an edge internal to $B$ since otherwise it would also contain an edge in $A \times B$. Hence, $G[A]$ or $G[B]$ (say $G[A]$) can be edge-partitioned into at most $\lfloor k/2 \rfloor$ complete bipartite graphs. Thus, the chromatic number of the induced subgraph on $A$ satisfies $\chi(G[A]) \leq f(\lfloor k/2 \rfloor)$. Also, $G[V \setminus A]$ can be edge-partitioned into at most $k$ complete bipartite graphs, so $\chi(G[V \setminus A]) \leq f(k)$. 
Finally, $\chi(G) \leq \chi(G[A])+\chi(G[V \setminus A]) \leq  f(k)+f(\lfloor k/2 \rfloor)$. 
\end{proof}

\begin{proof}[Proof of Theorem \ref{main}]
Let $G=(V,E)$ be a graph whose edge set can be partitioned into $k+1$ complete bipartite graphs. Let $H_i$ denote the $i^{\textrm{th}}$ complete bipartite graph of this edge partition and $A_i,B_i$ denote its two parts. If $H_i$ has an edge with vertices in a part of $H_j$, then $H_j$ cannot have an edge with vertices in a part of $H_i$. Indeed, if $(a_i,b_i) \in A_i \times B_i$ is an edge of $H_i$ in a part of $H_j$ (say $A_j$), and $(a_j,b_j) \in A_j \times B_j$ is an edge of $H_j$ in a part of $H_i$ (say $A_i$), then $(b_j,b_i) \in A_i \times B_i$ and $(b_i,b_j) \in A_j \times B_j$, so $(b_i,b_j)$ is an edge of both $H_i$ and $H_j$, contradicting that these complete bipartite graphs are edge-disjoint. 

Construct the auxiliary directed graph $D$ with vertex set $[k+1]$ with an edge $i \rightarrow j$  if the vertices of an edge of $H_i$ are in a part of $H_j$. By the above discussion, $D$ is an oriented graph, and so it has a vertex $j$ of indegree at most $k/2$. As in the proof of Proposition \ref{easier}, for each $i$, $H_i$ can have an edge in at most one of the two parts of $H_j$. Hence, $G[A_j]$ or $G[B_j]$ (say $G[A_j]$) can be edge-partitioned into at most $\lfloor k/4 \rfloor$ complete bipartite graphs. It follows that $\chi(G) \leq \chi(G[V \setminus A_j])+\chi(G[A_j]) \leq f(k)+f(\lfloor k/4 \rfloor)$. 
\end{proof}


\begin{thebibliography}{10}


\bibitem{BBGJK} K. Balodis, S. Ben-David, M. G\"o\"os, S. Jain, and R.  Kothari, Unambiguous DNFs and Alon-Saks-Seymour, {\it SIAM J. Comput.}, Special Section FOCS 2021 (2023), FOCS21-157--FOCS21-173.

\bibitem{GMNS} Z. Gao, B. D. McKay, R. Naserasr, and B. Stevens, Bipartite edge partitions and the former Alon-Saks-Seymour conjecture, {\it Australas. J. Combin.} {\bf 66} (2016), 211--228.

\bibitem{Goos} M. G\"o\"os, Lower bounds for clique vs. independent set, in {\it FOCS 2015}, pp. 1066--1076. 



\bibitem{HS} H. Huang and B. Sudakov, A counterexample to the Alon-Saks-Seymour conjecture and related problems, {\it Combinatorica} {\bf 32} (2012), 205--219.

\bibitem{MV} D. Mubayi and S. Vishwanathan, Bipartite coverings and the chromatic number, {\it Electron. J. Combin.} {\bf 16} (2009), Note 34, 5 pp.


\end{thebibliography}
\end{document}